\documentclass[12pt,reqno]{amsart}
\usepackage{geometry}
\usepackage{fullpage}
\usepackage{amssymb}
\usepackage{amsmath}
\usepackage{amsfonts}
\usepackage{amsthm}
\usepackage{mathrsfs}
\usepackage{graphicx}
\usepackage[table]{xcolor}
\usepackage{xcolor}
\usepackage{enumitem}
\usepackage{url}
\usepackage{tabularx}
\usepackage{colonequals}
\usepackage{hyperref}
\usepackage{multirow}
\hypersetup{colorlinks=true,urlcolor=blue,citecolor=blue,linkcolor=blue}
\usepackage{numprint} 
\newcommand{\defi}[1]{\emph{#1}} 

\setenumerate{label={$($\normalfont\arabic*)}}
\DeclareMathOperator{\Res}{Res}
\DeclareMathOperator{\spin}{spin}
\DeclareMathOperator{\mytr}{tr}
\DeclareMathOperator{\M}{M}
\DeclareMathOperator{\GL}{GL}
\DeclareMathOperator{\GSp}{GSp}
\DeclareMathOperator{\Gal}{Gal}
\DeclareMathOperator{\Jac}{Jac}
\DeclareMathOperator{\Frob}{Frob}
\DeclareMathOperator{\new}{new}
\DeclareMathOperator{\Nm}{Nm}
\DeclareMathOperator{\End}{End}
\DeclareMathOperator{\Aut}{Aut}

\DeclareMathOperator{\Disc}{Disc}
\def \N{\mathfrak{N}}
\newcommand{\frakN}{\N}
\newcommand{\frakp}{\mathfrak{p}}
\def \PP{\mathbb{P}}
\def \QQ{\mathbb{Q}}
\def \Q{\QQ}
\def \QA#1#2#3{\displaystyle{\left(\frac{#1,#2}{#3}\right)}}
\def \ZZ{\mathbb{Z}}
\def \FF{\mathbb{F}}
\def \RR{\mathbb{R}}
\def\calO{\mathcal{O}}
\def\CC{\mathbb{C}}
\def\Cprime{C}
\def\Czero{C_0}

\newcommand{\lmfdbfield}[1]{\href{https://www.lmfdb.org/NumberField/#1}{\textsf{#1}}}

\newcommand{\lmfdbbmfshort}[3]{\href{https://www.lmfdb.org/ModularForm/GL2/ImaginaryQuadratic/#1/#2/#3}{\textsf{#2-#3}}}
\newcommand{\lmfdbbmf}[3]{\href{https://www.lmfdb.org/ModularForm/GL2/ImaginaryQuadratic/#1/#2/#3}{\textsf{#1-#2-#3}}}
\newcommand{\lmfdbbmfspace}[2]{\href{https://www.lmfdb.org/ModularForm/GL2/ImaginaryQuadratic/#1/#2}{\textsf{#1-#2}}}
\newcommand{\lmfdbcmf}[4]{\href{https://beta.lmfdb.org/ModularForm/GL2/Q/holomorphic/#1/#2/#3/#4/}{\textsf{#1.#2.#3.#4}}}
\newcommand{\lmfdbec}[3]{\href{https://www.lmfdb.org/EllipticCurve/Q/#1/#2/#3}{\textsf{#1.#2#3}}}
\newcommand{\lmfdbecisog}[2]{\href{https://www.lmfdb.org/EllipticCurve/Q/#1/#2}{\textsf{#1.#2}}}

\newcommand{\lmfdbecnfisog}[3]{\href{https://www.lmfdb.org/EllipticCurve/#1/#2/#3}{\textsf{#1.#2-#3}}}

\newcommand{\lmfdbDC}[2]{\href{https://www.lmfdb.org/Character/Dirichlet/#1/#2}{$\chi_{#1}(#2,\cdot)$}}

\def \Pari{{\sc PARI/GP}}
\def \Magma{{\sc Magma}}

\newcommand{\QQbar}{\QQ^{\textup{al}}}
\newcommand{\Qbar}{\QQbar}

\newcommand{\Ggroup}{H}

\newcommand{\Abar}{A^{\textup{al}}}

\theoremstyle{plain}
\newtheorem{theorem}[equation]{Theorem}
\newtheorem{proposition}[equation]{Proposition}

\newtheorem{conjecture}[equation]{Conjecture}
\newtheorem*{theorem1}{Theorem A}

\theoremstyle{remark}
\newtheorem{remark}[equation]{Remark}

\numberwithin{equation}{section}

\newenvironment{enumalph}
{\begin{enumerate}}
{\end{enumerate}}

\begin{document}
\title[Bianchi newforms and QM]{On rational Bianchi newforms and abelian surfaces with
  quaternionic multiplication}

\author[Cremona]{John Cremona}
\address{Warwick Mathematics Institute, Zeeman Building, University of Warwick, Coventry, CV4 7AL, UK}
\email{j.e.cremona@warwick.ac.uk}

\author[Demb\'el\'e]{Lassina Demb\'el\'e}
\address{Department of Mathematics, Dartmouth College, 6188 Kemeny Hall, Hanover, NH 03755, USA}
\email{lassina.dembele@gmail.com}

\author[Pacetti]{Ariel Pacetti} \address{FAMAF-CIEM, Universidad Nacional de
  C\'ordoba. C.P:5000, C\'ordoba, Argentina.}
\email{apacetti@famaf.unc.edu.ar}

\author[Schembri]{Ciaran Schembri}
\address{Department of Mathematics, Dartmouth College, 6188 Kemeny Hall, Hanover, NH 03755, USA}
\email{ciaran.schembri@dartmouth.edu}
\thanks{}

\author[Voight]{John Voight}
\address{Department of Mathematics, Dartmouth College, 6188 Kemeny Hall, Hanover, NH 03755, USA}
\email{jvoight@gmail.com}
\urladdr{\url{http://www.math.dartmouth.edu/~jvoight/}}


\begin{abstract}
We study the rational Bianchi newforms (weight $2$, trivial character,
with rational Hecke eigenvalues) in the LMFDB that are not associated
to elliptic curves, but instead to abelian surfaces with quaternionic
multiplication.  Two of these examples exhibit a rather special kind
of behaviour: we show they arise from twisted base change of a
classical newform with nebentypus character of order~$4$ and eight
inner twists.
\end{abstract}

\maketitle

\section{Introduction}\label{sec:intro}

Let $f$ be a classical newform with weight $2$ and level $N$.  When
$f$ has rational coefficients, the Eichler--Shimura construction
attaches to $f$ an elliptic curve $E_f$ defined over~$\QQ$, a quotient
of the Jacobian of the modular curve $X_0(N)$, such that
$L(f,s)=L(E_f,s)$.  More generally, if $f$ has coefficients in a
number field $M_f$ of degree~$d$, this construction associates a
simple abelian variety~$A_f$ of dimension~$d$ defined over~$\QQ$, now
a quotient of the Jacobian of the modular curve $X_1(N)$, such that
$\End(A_f) \otimes \QQ \simeq M_f$ and $L(A_f,s)=\prod_\sigma
L(\sigma(f),s)$, the product over all embeddings $\sigma\colon M_f
\hookrightarrow \CC$.  (It may happen that $A_f$ acquires additional
endomorphisms over a number field, and in particular $A_f$ may not be
geometrically simple.)

It is natural to seek such a relationship for
modular forms over other number fields.  
In the case of totally real number fields, in many instances 
(for example, over fields of odd degree) a
construction analogous to the Eichler--Shimura construction is
obtained by replacing the modular curve by a
suitable Shimura curve.  However, a general such construction 
over number fields is not known.  
So with the simplest case in mind, we are led to ask: given a Bianchi 
newform $F$ of weight $2$ over an imaginary quadratic field $K$,
is there an abelian variety $A_F$ defined over $K$ whose
$L$-series matches that of $F$ (as above)?  

A precise answer to this question, consistent with the predictions of
the Langlands philosophy, is provided affirmatively by a conjecture of
Taylor \cite[Conjecture 3]{MR1403943}---but with a wrinkle.  Let $M_F$
be the number field generated by the Hecke eigenvalues of $F$, and let
$d=[M_F:\QQ]$.  Then conjecturally there is an abelian variety $A_F$
of dimension $2d$ defined over $K$ with \emph{quaternionic
  multiplication} (QM) by a quaternion algebra $B_F$ over $M_F$ (and
defined over $K$) such that
\begin{equation} 
L(A_F,s)=\prod_{\sigma:M_F \hookrightarrow \CC} L(\sigma(F),s)^2. 
\end{equation}
In this formulation, it may happen that the quaternion algebra $B_F \simeq \M_2(M_F)$ is split, in which case $A_F \sim (A_F')^2$ is isogenous over $K$ to the square of an abelian variety $A_F'$ of dimension $d$ and $L(A_F',s)=\prod_\sigma L(\sigma(F),s)$.  (In general, if the base field $K$ has a real place, then $B_F$ is necessarily split, which explains the simplifications above.)  

In particular, if $F$ has rational coefficients (i.e., $M_F=\QQ$), then attached to $F$ is conjecturally a QM abelian surface that is often the square of an elliptic curve.  Examples of the phenomenon where the quaternion algebra is a division algebra were exhibited by Cremona \cite{JChyp,JCextratwist} in the case where $F$ is the base-change of a classical newform with quadratic coefficients, or a quadratic twist of such a form.  More recently, Schembri \cite{SchembriQM} exhibited examples which are not base change.  QM abelian surfaces have been called \emph{false generalized elliptic curves} by Taylor \cite{MR1403943}, \emph{false elliptic curves} by Serre (as attributed by Deligne--Rapoport \cite[p.\ DeRa-12]{MR0337993}), and \emph{fake elliptic curves} by others.
Perhaps this terminology would be more appropriate for the non-existent geometric object whose $L$-function is equal to $L(F,s)$.

Given that the situation is rather unusual, in this paper we study it
explicitly.  We consider the Bianchi newforms $F$ of weight $2$ with
rational coefficients in the \emph{$L$-Functions and Modular Forms
  Database} (LMFDB) \cite{lmfdb}.  For thousands of such forms, we can explicitly attach an elliptic curve $E_F$ such that $L(F,s)=L(E_F,s)$, 
with $E_F$ unique up to isogeny over $K$.  For the remaining forms,
we can find no such elliptic curve---in Section~\ref{sec:others}, we
list these forms and summarize what is known about them.  In all these
remaining examples, we can attach a QM abelian surface $A$ over $K$
such that $L(A,s)=L(F,s)^2$, in agreement with the above conjecture.

Two examples in the list stand out, and because of their exotic
properties they command the majority of our attention here.  We
consider the rational Bianchi newform $F$ of weight $2$ with LMFDB
label\footnote{Here, \lmfdbfield{2.0.7.1} is the LMFDB label for the
  base field $K=\QQ(\sqrt{-7})$ and \textsf{30625.1} the label for the
  level ideal~$(175)$, which has norm~$30625$.  The final {c} is the
  alphabetic label for this specific newform at that level. We use
  either full labels such as \lmfdbbmf{2.0.7.1}{30625.1}{c} for
  Bianchi newforms, or the shorter version
  \lmfdbbmfshort{2.0.7.1}{30625.1}{c} which omits the field when that
  is clear from the context.} \lmfdbbmf{2.0.7.1}{30625.1}{c} over
$K=\QQ(\sqrt{-7})$, along with a certain quadratic twist~$G$, which
has label \textup{\lmfdbbmf{2.0.7.1}{30625.1}{e}}. The newform $F$ has
rational coefficients and is a twist of a base-change form, but there
is no \emph{quadratic} twist of~$F$ which is base-change; in fact the
simplest twist which is base-change is by a character of order~$8$, so
has neither trivial nebentypus nor rational coefficients.
  
  Our main result (combining Proposition \ref{prop:rational-twist}, Theorem~\ref{T:4fold}, and Proposition~\ref{prop:FaltSer}) is the
  following.

  \begin{theorem1}
Let $F$ be the rational Bianchi newform over $K \colonequals \QQ(\sqrt{-7})$ of weight $2$, level~$(175)$, and trivial character with LMFDB label \textup{\lmfdbbmf{2.0.7.1}{30625.1}{c}}.
Then the following statements hold.
\begin{enumalph}
\item There exists a classical modular form $f$ of weight $2$, level $1225$, character of conductor $35$ and order $4$, and Hecke field $\Q(\zeta_{24})$, such that 
\[ F = f_K \otimes \psi, \]
where $f_K$ is the base change of $f$ to $K$ and $\psi$ is a character of conductor $(5\sqrt{-7})$ and order $8$.
\item There exists an explicit genus $2$ curve $X_F$ over $K$ whose Jacobian $J_F \colonequals \Jac(X_F)$ has QM by a quaternion algebra of discriminant $15$ and
\[ L_\frakp(X_F,T)=L_\frakp(J_F,T)=L_\frakp(F,T)^2 \]
for all primes $\frakp \neq (2),(5),(\sqrt{-7})$.  Moreover, all endomorphisms of $J_F$ are defined over $K$.
\item Let $A_f$ be the abelian eightfold over $\Q$ attached to $f$.
  Then $\End(A_f)=\ZZ[\zeta_{24}]$.  Let $A' \colonequals
  (A_f)_K^{(\psi)}$ be the base change of $A$ to $K$ twisted by $\psi$
  (relative to an automorphism of $A_f$ of order $8$).  Then there is
  an isogeny $A' \sim J_F^2 \times J_G^2$ over $K$, where $G=F \otimes
  \psi^4$ is a quadratic twist of $F$ (with LMFDB label
  \textup{\lmfdbbmf{2.0.7.1}{30625.1}{e}}), and $J_G$ the corresponding
  quadratic twist of $J_F$.
\end{enumalph}
 \label{theorem:main}
\end{theorem1}

The form is described in detail in Section~\ref{sec:overQ}, and the
twists are given explicitly in Section~\ref{sec:twisting}.  An
equation for the curve $X_F$ in Theorem \ref{theorem:main}(b) is given
in \eqref{eqn:woahnelly}.

\subsection*{Contents}

The paper is organized as follows.  In Section~\ref{sec:others} we
survey Bianchi newforms in the LMFDB.  After quickly setting up
notation for characters in Section~\ref{sec:characters}, we describe
the classical modular form $f$ in Section~\ref{sec:overQ}; then we
consider its base change and twists in
Sections~\ref{sec:basechange}--\ref{sec:twisting}.  We then give a
model for the curve $X_F$ in Section~\ref{sec:genus2curve} and in
Section~\ref{sec:others175} briefly describe related Bianchi modular
forms in the same level.  We conclude in Section~\ref{sec:paramodular}
with some connections to the Paramodular Conjecture.

\subsection*{Acknowledgments} 

The authors would like to thank Frank Calegari for useful discussions
and suggestions and the anonymous referees for helpful feedback.
Many thanks also to David Loeffler for directing our
attention to the sources in Remark \ref{rmk:loeffler}.  Cremona was
supported by EPSRC Programme Grant EP/K034383/1 \textit{LMF:
  L-Functions and Modular Forms}, and the Horizon 2020 European
Research Infrastructures project \textit{OpenDreamKit}
(\#676541). Pacetti was partially supported by grant PICT-2018-02073. Demb\'el\'e, Schembri, and Voight were supported by a Simons Collaboration Grant
(550029; to Voight).

\section{Rational Bianchi newforms with no
  associated elliptic curve}\label{sec:others} 
  
The LMFDB \cite{lmfdb} currently contains rational Bianchi newforms
over each of the five fields $K=\QQ(\sqrt{d})$ for
$d=-3,-4,-7,-8,-11$, as computed by Cremona \cite{JChyp}, of all
levels of norm up to a bound depending on the field, currently (August
2020): $\numprint{150000}$ for $d=-3$, $\numprint{100000}$ for $d=-4$,
and $\numprint{50000}$ for $d=-7,-8,-11$.  The total number of
rational newforms within these ranges is~$\numprint{185581}$, of which
$\numprint{5917}$ arise from base change from $\Q$.  All but~$40$ of
the $\numprint{185581}$ newforms~$F$ have a candidate associated
elliptic curve~$E$ defined over the field~$K$ in the sense given in
the introduction, namely $L(E,s)=L(F,s)$; or equivalently that they
have the same compatible system of $\ell$-adic Galois representations.
It is possible (in principle) to prove modularity of all these
elliptic curves, for example using the Faltings--Serre--Livn\'e method
(as detailed by Dieulefait--Guerberoff--Pacetti \cite{DGPmodularity}).
This is being carried out systematically by Mattia Sanna (Warwick) for
his PhD thesis \cite{SannaThesis}, using a combination of $2$-adic and
$3$-adic methods combined with recent modularity lifting theorems.

Of the 40 remaining rational Bianchi newforms without associated
elliptic curves, six arise from base change of classical modular forms
over~$\QQ$ with trivial character, real quadratic coefficients and
inner twists: see \cite{JCextratwist}.
A further 18 are twisted base change, and for 16 of these~18, the
base-change form itself and the form over~$\QQ$ have trivial
character; the other two will be studied in this paper.  The remaining
$16$ forms are associated to abelian surfaces over the base
field with quaternionic multiplication (QM), and were considered by
Schembri \cite{SchembriQM}.

\begin{small}
\begin{equation} \label{table:qmdata}\addtocounter{equation}{1} \notag
\begin{gathered}
\begin{tabular}{l|ccccc|c}
Field discriminant  & $-4$ & $-8$ & $-3$ & $-7$ & $-11$ & Total \\
  Level norm bound & $\numprint{100000}$ & $\numprint{50000}$ & $\numprint{150000}$ & $\numprint{50000}$ & $\numprint{50000}$ \\
  \hline\hline
  Rational newform counts\\
  \hline
All rational newforms & $\numprint{40030}$ & $\numprint{36843}$ & $\numprint{42343}$ & $\numprint{35824}$ & $\numprint{30541}$ & $\numprint{185581}$ \\
Not twist of base change &  \numprint{38844} & \numprint{35536} &
\numprint{40986} & \numprint{34830} & \numprint{29468} & \numprint{179664} \\
Base change of rational newform & 908 & 809 & 955 & 518 & 504 & \numprint{3694} \\
Base change of quadratic newform & 16 & 10 & 28 & 24 & 13 & \numprint{91} \\
\parbox[c][6ex]{36ex}{Twist of base change of \\ rational newform} & 240 & 350 & 328 & 378 & 492 & \numprint{1788} \\
\parbox[c][6ex]{36ex}{Twist of base change of \\ quadratic newform} & 22 & 138 & 46 & 72 & 64 & {342} \\
\parbox[c][6ex]{36ex}{Twist of base change of \\ higher-dimensional newform}  & 0 & 0 & 0 & 2 & 0 & {2} \\
\hline
\hline
Rational newforms with no elliptic curve\\
\hline
Total & $4$ & $8$ & $21$ & $7$ & $0$ & $40$ \\
Quadratic twist of base change & $0$ & $8$ & $9$ & $5$ & $0$ & $22$ \\
Higher order twist of base change & $0$ & $0$ & $0$ & $2$ & $0$ & $2$ \\
Not twist of base change & $4$ & $0$ & $12$ & $0$ & $0$ & $16$ \\
\end{tabular} \\
\text{Table \ref{table:qmdata}: Summary of Bianchi modular form data in the LMFDB}
\end{gathered}
\end{equation}
\end{small}

There is a natural action of $\Gal(K\,|\,\Q) = \langle \tau \rangle$
on rational Bianchi newforms, $f \mapsto \tau(f)$, where
$a_\frakp(\tau(f))=a_{\tau(\frakp)}(f)$; if $g=\tau(f)$ then we say that
$f,g$ are \defi{Galois conjugate}.  (Our forms have rational
coefficients, so there should be no confusion with other Galois
actions.)

Organizing according to field and level, we have the following further
details about the 40 rational newforms with no associated elliptic
curve.  These are either base-change, or come in Galois conjugate
pairs.  In some cases the forms mentioned are twists of base-change,
but the base-change newform itself has too large a level to be in the
current database, so does not yet have a label.

\begin{itemize}
\item $K=\QQ(\sqrt{-1})$, LMFDB label~$\lmfdbfield{2.0.4.1}$: four
  cases, none of which are base-change (even up to twist).
  \begin{itemize}
  \item \lmfdbbmfshort{2.0.4.1}{34225.3}{a} is not a base-change or twist
    of base-change.  It is associated to a simple abelian surface with
    QM over~$K$, the Jacobian of the
    genus~$2$ curve~$C_1$ found by Schembri \cite[Theorem 1.1]{SchembriQM}.
    \item \lmfdbbmfshort{2.0.4.1}{34225.3}{b} is a quadratic twist of
      \lmfdbbmfshort{2.0.4.1}{34225.3}{a}.
    \item \lmfdbbmfshort{2.0.4.1}{34225.7}{a} and
      \lmfdbbmfshort{2.0.4.1}{34225.7}{a} are Galois conjugates of the
      previous two.
  \end{itemize}
\item $K=\QQ(\sqrt{-2})$, LMFDB label~$\lmfdbfield{2.0.8.1}$: eight cases, all base-change
  (up to twist).
  \begin{itemize}
  \item \lmfdbbmfshort{2.0.8.1}{5625.1}{b} and its Galois conjugate
    \lmfdbbmfshort{2.0.8.1}{5625.3}{b} are both twists of base-changes of
    classical newforms in~$S_2(28800)$.
    \item \lmfdbbmfshort{2.0.8.1}{6561.5}{a} and its Galois conjugate
      \lmfdbbmfshort{2.0.8.1}{6561.5}{d} are both twists of a Bianchi
      newform at level $(2592)$, which is a base-change from
      $S_2(20736)$. The classical newform has coefficients in
      $\QQ(\sqrt{15})$ and attached to the Bianchi newform is an
      abelian surface with QM of discriminant~$10$, but an explicit
      equation is not known.

    \item \lmfdbbmfshort{2.0.8.1}{21609.3}{b} and
      \lmfdbbmfshort{2.0.8.1}{21609.3}{c} are both twists of a Bianchi
      newform at level $(14112)$, which is a base-change from
      $S_2(11289)$. The classical newform has coefficients in
      $\QQ(\sqrt{7})$ and attached to the Bianchi newform is an
      abelian surface with QM of discriminant~$14$.  Again, an
      explicit equation is not known.
    \item \lmfdbbmfshort{2.0.8.1}{21609.1}{b} and
      \lmfdbbmfshort{2.0.8.1}{21609.1}{c} are Galois conjugates of the
      previous two.
  \end{itemize}
\item $K=\QQ(\sqrt{-3})$, LMFDB label~$\lmfdbfield{2.0.3.1}$: $21$
  cases, of which 12 are not twists of base-change but have associated
  genus~$2$ curves with QM:
  \begin{itemize}
  \item \lmfdbbmfshort{2.0.3.1}{61009.1}{a} is associated to the genus~$2$
    curve~$C_2$ found by Schembri  \cite[Theorem 1.1]{SchembriQM}, while
    \lmfdbbmfshort{2.0.3.1}{61009.1}{b} is a quadratic twist of this, and
    \lmfdbbmfshort{2.0.3.1}{61009.9}{a} and \lmfdbbmfshort{2.0.3.1}{61009.9}{b}
    are their Galois conjugates.
  \item \lmfdbbmfshort{2.0.3.1}{67081.3}{a} is associated to the genus~$2$
    curve~$C_3$ found by Schembri \cite[Theorem 1.1]{SchembriQM}, while
    \lmfdbbmfshort{2.0.3.1}{67081.3}{b} is a quadratic twist of this, and
    \lmfdbbmfshort{2.0.3.1}{67081.7}{a} and \lmfdbbmfshort{2.0.3.1}{67081.7}{b}
    are their Galois conjugates.
  \item \lmfdbbmfshort{2.0.3.1}{123201.1}{b} is associated to the genus~$2$
    curve~$C_4$ found by Schembri \cite[Theorem 1.1]{SchembriQM}, while
    \lmfdbbmfshort{2.0.3.1}{123201.1}{c} is a quadratic twist of this, and
    \lmfdbbmfshort{2.0.3.1}{123201.3}{b} and \lmfdbbmfshort{2.0.3.1}{123201.3}{c}
    are their Galois conjugates.
  \end{itemize}
  The other nine cases are twists of base-changes of forms with
  trivial character, coefficients in a real quadratic field and an
  inner twist; in these cases, explicit equations are not known:
  \begin{itemize}
  \item \lmfdbbmfshort{2.0.3.1}{5625.1}{a} is the base-change of
    \lmfdbcmf{225}{2}{a}{f}, which has CM by $-15$, coefficients
    in~$\QQ(\sqrt{5})$, and an inner twist.  This is an example of the
    phenomenon first described by Cremona \cite{JCextratwist}.
    \item \lmfdbbmfshort{2.0.3.1}{6561.1}{b} is base-change of
      \lmfdbcmf{243}{2}{a}{d}, with coefficients in~$\QQ(\sqrt{6})$
      and inner twist, attached to an abelian surface with QM of
      discriminant~$6$.
    \item \lmfdbbmfshort{2.0.3.1}{30625.1}{a} and its Galois conjugate
      \lmfdbbmfshort{2.0.3.1}{30625.3}{a} are twists of base-change of
      forms in~$S_2(11025)$ with CM by~$-15$.
    \item \lmfdbbmfshort{2.0.3.1}{50625.1}{c} and its Galois conjugate
      \lmfdbbmfshort{2.0.3.1}{50625.1}{d} are base-changes of
      \lmfdbcmf{675}{2}{a}{l} and \lmfdbcmf{675}{2}{a}{m}
      respectively, with coefficients in~$\QQ(\sqrt{2})$ and inner
      twist,  attached to abelian surfaces with QM of discriminant~$6$.
    \item \lmfdbbmfshort{2.0.3.1}{65536.1}{b} and its Galois conjugate
      \lmfdbbmfshort{2.0.3.1}{65536.1}{e} are twists of a form at
      level~$(768)$, which is base-change of \lmfdbcmf{2304}{2}{a}{p},
      with CM by~$-6$, coefficients in~$\QQ(\sqrt{2})$ and inner
      twist.
    \item \lmfdbbmfshort{2.0.3.1}{104976.1}{a} is base-change of
      \lmfdbcmf{972}{2}{a}{e}, with coefficients in~$\QQ(\sqrt{2})$
      and inner twist, attached to an abelian surface with QM of
      discriminant~$6$.
  \end{itemize}
\item $K=\QQ(\sqrt{-7})$, LMFDB label~$\lmfdbfield{2.0.7.1}$: seven
  cases.  Of these, two are the forms~$F$ and~$G$ in Theorem
  \ref{theorem:main} and discussed in detail in
  Section~\ref{sec:others175} below.  In addition, we have
  \begin{itemize}
  \item  \lmfdbbmfshort{2.0.7.1}{30625.1}{d}: see Section~\ref{sec:others175}.
  \item \lmfdbbmfshort{2.0.7.1}{40000.1}{b}  and its Galois conjugate
    \lmfdbbmfshort{2.0.7.1}{40000.7}{b} are twists of
    \lmfdbbmfshort{2.0.7.1}{30625.1}{d}.
  \item \lmfdbbmfshort{2.0.7.1}{10000.1}{b} and its Galois conjugate
    \lmfdbbmfshort{2.0.7.1}{10000.5}{b} are twists of forms which are
    the base-change of forms in~$S_2(19600)$ with coefficients
    in~$\QQ(\sqrt{5})$ and inner twist.
  \end{itemize}
\item $K=\QQ(\sqrt{-11})$, LMFDB label~$\lmfdbfield{2.0.11.1}$: none.
\end{itemize}

\section{Characters} \label{sec:characters}

We begin by setting some notation, in particular we define some characters.
Let $\QQbar$ be an algebraic closure of $\QQ$.  
By a \defi{character} over $\QQ$ we will mean either a Dirichlet character or the character of
the absolute Galois group $\Gal_\Q \colonequals \Gal(\QQbar\,|\,\QQ)$ attached to it via class field theory;
we define the same over a number field $K$ allowing Hecke 
characters (of finite order).  If $\chi$ is a character over~$\QQ$, we may
talk of its \defi{restriction} $\left.\chi\right|_K$ to~$K$
thinking of it as a character of $\Gal_\Q$ and restricted to $\Gal_K \colonequals \Gal(\QQbar\,|\,K)$.

Let $N \colonequals 1225=5^2 7^2$; we now define some Dirichlet
characters of modulus $N$.  Denote by $\chi_5$ the mod~$5$ cyclotomic
character of order~$4$ (with LMFDB label \lmfdbDC{1225}{393}%
) and by $\chi_{-7}$ the quadratic
character of conductor~$7$ (label \lmfdbDC{1225}{1126}%
).  As characters of $\Gal_\Q$, these cut out
the cyclic extensions $\QQ(\zeta_5)$ and
\begin{equation}
K \colonequals \QQ(\sqrt{-7})
\end{equation}
respectively.  Define $\varepsilon \colonequals \chi_5\chi_{-7}$
(label \lmfdbDC{1225}{293}%
);
then $\varepsilon$ has conductor~$35$ and order~$4$.  Let
\begin{equation}
\Ggroup \colonequals \langle \chi_5 \rangle \times \langle \chi_{-7} \rangle \simeq \ZZ/4\ZZ \times \ZZ/2\ZZ
\end{equation} 
be the group generated by
$\chi_{-7}$ and $\chi_5$.

After restricting to~$K$, the character $\chi_{-7}|_K$ is trivial,
so $\varepsilon|_K=\chi_5|_K$.
Also, while $\chi_5$ is not a square (since $\QQ$ has no cyclic
extension of degree~$8$ ramified only at~$5$), its restriction to~$K$
is a square, since $5$ is inert in~$K$.  We
denote by $\psi$ one of the two order eight characters of~$K$ such that
$\psi^{-2} = \left.\chi_5\right|_K = \left.\varepsilon\right|_K$; the
other such character is $\psi^5$.  The conductor of~$\psi$ is
$(5\sqrt{-7})$.

The reason why the character $\psi$ must ramify at the prime
$\sqrt{-7}$ can be explained as follows.  The multiplicative group of
the residue field of the prime ideal $(5)$ is $\FF_{25}^\times$ and so
admits an order $8$ character $\psi_5$. Such a character is odd, i.e.,
$\psi_5(-1) = -1$.  A finite order Hecke character has infinity type
is $(0,0)$ (i.e. it is trivial at the Archimedean places), hence if it
is unramified outside $5$, such a character would not be trivial at
$-1$. However, the quadratic character $\psi_{\sqrt{-7}}$ is also odd,
and the character $\psi$ corresponds to the Hecke character ramified
at $\{5, \sqrt{-7}\}$ with such ramification at both primes.

\section{A classical newform of level 1225 and an abelian eightfold}\label{sec:overQ}

In this section, we investigate a certain special classical newform;
all statements below can be confirmed using \Magma\ \cite{Magma} or \Pari\ \cite{PARI2}.

We retain notation from the previous section.  Consider the newspace
$S_2(N,\varepsilon)^{\new}$.  Its relative dimension as
a vector space over~$\QQ(\varepsilon) \simeq \Q(i)$ is~$56$, and it decomposes
under the action of the Hecke algebra into~$7$ irreducible components,
of relative dimensions $2, 2, 4, 8, 12, 12, 16$.  Their absolute dimensions, as
$\QQ$-vector spaces, are obtained by multiplying by 
$[\QQ(\varepsilon):\QQ]=2$.

Let $f$ be any one of the four Galois conjugate newforms in the unique
$4$-dimensional component; these four form a single orbit under
$\Gal(\Qbar\,|\,\QQ(\varepsilon))$.  Of
the eight $\Gal(\Qbar\,|\,\QQ)$-conjugates of~$f$, four are in
$S_2(N,\varepsilon)$ and the other four in
$S_2(N,\overline{\varepsilon})$, where $\overline{\varepsilon} =
\varepsilon^3 = \overline{\chi_5}\chi_{-7}$.  The form $f$ does 
not have CM and its Hecke field is $M_f \colonequals \QQ(\zeta_{24})$.

\begin{remark}
The form $f$ has $Nk^2 = 4900  \geq 4000$, and so is just outside the current range of forms in the LMFDB.
\end{remark}

\begin{proposition}\label{prop:inner-twist}
 For each character $\chi\in \Ggroup$, the twist $f\otimes\chi$ is a
 conjugate of~$f$, i.e., there exists $\sigma \in \Gal(M_f\,|\,\Q)$ such that $f \otimes \chi = \sigma(f)$.  
\end{proposition}
\begin{proof}
First observe that the level~$N$ is divisible by the square of the
conductor of~$\varepsilon$, and also that neither $\chi_{-7}$ nor
$\chi_5$ is a square.  This implies that the twists have the same level
as~$N$ \cite[Proposition 3.64]{Shimura-book}.

Let $\chi \in \Ggroup$.  Then the nebentypus character
of~$f\otimes\chi$ is $\varepsilon \chi^2$, which equals either
$\varepsilon$ or~$\overline{\varepsilon}$.  Hence the twists are
newforms in either $S_2(N,\varepsilon)^{\new}$ or
$S_2(N,\overline{\varepsilon})^{\new}$.  The $\Gal_\Q$-orbit of $f \otimes \chi$ 
has size at most~$8$ (but might \textit{a priori} be smaller).  However,
$a_2(f)=-\zeta_{24}^3$ is an eighth root of unity, while the $a_2$
coefficients of all the newforms of smaller dimension in
$S_2(N,\varepsilon)$ (which lie in $\QQ(\zeta_{12})$) are not roots of
unity, so cannot be equal to $\chi(2)a_p(f)$.
Hence $f \otimes \chi$ has the same Hecke field as~$f$ and by the
uniqueness of the $\QQ$-dimension~$8$ component of
$S_2(N,\varepsilon)^{\new}$ must be conjugates of~$f$.
\end{proof}

Proposition \ref{prop:inner-twist} says that the newform $f$ has 
eight \emph{inner twists} $(\sigma, \chi)$ (including the
trivial twist, where both $\chi$ and $\sigma$ are trivial).  
To keep track of these inner twists, we use the following notation:
let $\sigma_a \in \Gal(\Q(\zeta_{24})\,|\,\Q)$ be the
automorphism such that $\sigma_a(\zeta_{24})=\zeta_{24}^a$.  
Checking the first few Fourier coefficients $a_p(f)$ we see that the
inner twists are defined by the data $(\sigma, \chi)$ where:
\begin{equation} \label{table:innertwists}\addtocounter{equation}{1} \notag
\begin{gathered}
\begin{tabular}{c|cccccccc}
$\chi$ & 1 & $\chi_5$ & $\chi_5^2$ & $\chi_5^3$ & $\chi_{-7}$ & $\chi_5\chi_{-7}$ & $\chi_5^2\chi_{-7}$ &
$\chi_5^3\chi_{-7}$ \\
\hline
$\sigma$ & $\sigma_1$ & $\sigma_{19}$ & $\sigma_{13}$ & $\sigma_{7}$ & $\sigma_{17}$ & $\sigma_{11}$ & $\sigma_{5}$ & $\sigma_{23}$ \\
\end{tabular} \\
\text{Table \ref{table:innertwists}: Inner twists of $f$}
\end{gathered}
\end{equation}

By the Eichler--Shimura construction, attached to (the Galois 
orbit of) $f$ is a simple abelian $8$-fold $A=A_f$, defined over~$\QQ$ 
with $\End(A)_\Q \colonequals \End(A) \otimes \Q \simeq M_f$.  
Write
$A_K \colonequals A \times_\Q K$ for the base extension
of $A$ to $K=\Q(\sqrt{-7})$, and
$\Abar \colonequals A \times_\Q \Qbar$ for the base extension
of $A$ to $\Qbar$.  To determine how $A$ splits
over~$\Qbar$ we must determine the \emph{geometric} endomorphism 
algebra~$\End(A_{\Qbar})_\Q$, which is generated
over~$\QQ(\zeta_{24})$ by endomorphisms arising from inner twists
of~$f$, by Ribet--Pyle \cite{Ribet-Twists,Ribet-Abvars,Pyle} theory.  We use the implementation in \Magma\ 
\cite{Magma} due to Quer \cite{Quer-BuildingBlocks} and Stein.  

\begin{proposition} \label{prop:ribetpyle}
The following statements hold.
\begin{enumalph}
\item Each inner twist
$(\sigma, \chi)$ of $f$ gives rise to an endomorphism $\xi_{(\sigma, \chi)}$
of~$\Abar$, defined over the field $\Q(\chi)$ cut out by~$\chi$ and satisfying
\begin{equation}
\alpha\xi_{(\sigma, \chi)}=\xi_{(\sigma, \chi)}\sigma(\alpha)
\end{equation} 
for all $\alpha\in\QQ(\zeta_{24})$.  The endomorphisms $\xi_{(\sigma, \chi)}$ are a 
$\QQ(\zeta_{24})$-basis for $\End(\Abar)_\Q$.
\item We have $\End(A_K)_\Q \simeq \M_2(\QQ(\zeta_8))$, with centre $Z(\End(A_K))=\ZZ[\zeta_8]$.
\item We have $\End(\Abar)_{\Q} \simeq \M_4(B)$ where $B \simeq \QA{3}{5}{\QQ}$ is the division quaternion algebra of discriminant $15$ over $\Q$.  All endomorphisms of~$\Abar$ are defined
over~$K' \colonequals \QQ(\zeta_5,\sqrt{-7})$.
\item There exists a simple abelian surface $S$ defined over $K$ with $\End(S)_\Q \simeq B$ and an isogeny $A \sim S^4$ defined over $K'$.
\end{enumalph}
\end{proposition}

\begin{proof}
Statement (a) is an application of work of Ribet \cite[(5.5)]{Ribet-Twists} together with the computation of inner twists above (Proposition \ref{prop:inner-twist}).  
Statement (b) follows from (a): over $K=\QQ(\sqrt{-7})$, the abelian variety $A_K$ obtains the endomorphism $\xi_{(\sigma_{17}, \chi_{-7})}$; since the fixed field of $\sigma_{17}$ is $\QQ(\zeta_8)$ and
$\QQ(\zeta_{24})=\QQ(\zeta_8)(\sqrt{-3})$, we conclude that
\begin{equation} 
\End(A_K)_\Q \simeq \QA{-3}{-7}{\QQ(\zeta_8)} \simeq \M_2(\QQ(\zeta_8)).  
\end{equation}
We have $\zeta_8 \in \End(A_K)$ because already $\ZZ[\zeta_{24}] \subseteq \End(A)$, and $\Q(\zeta_8) \subseteq Z(\End(A_K)_\Q)$.  

We finish (c) and (d) following Pyle \cite{Pyle}.  Over $\Qbar$, we have $A\sim S^r$ where $S$ (over $\Qbar)$ is simple, called a
\emph{building block}.
We have $\dim S=1$ or~$2$, and these cases are distinguished by a Brauer class; \Magma\ confirms
that $\dim S=2$, $r=4$, and therefore $\End(S)_\Q \simeq B$.  
Moreover, calculating the obstruction to descent \cite[Proposition 5.2]{Pyle},
we find that $S$ and its endomorphism algebra may be defined, 
up to isogeny, over $K$.
\end{proof}

The inner twist relation $f\otimes\chi=\sigma(f)$ implies, on comparing
the nebentypus character on both sides, that
\begin{equation}
\varepsilon\chi^2=\sigma_\chi(\varepsilon)
\end{equation} 
and that for all but finitely many $p$, we have 
\begin{equation} 
a_p(f)\chi(p) = \sigma(a_p(f)). 
\end{equation}
Together these imply that $a_p(f)^2/\varepsilon(p)$ is fixed by
all~$\sigma\in\Gal(\QQ(\zeta_{24})/\QQ)$, so we conclude that
$a_p(f)^2/\varepsilon(p)\in\QQ$ for all but finitely many $p$
(and this relation holds in generality).

The $L$-function of~$A$ (over $\QQ$) is
\begin{equation}
L(A,s) = \prod_{\chi \in \Ggroup}L(f\otimes\chi,s),
\end{equation}
since the $8$ Galois conjugates of~$f$ are precisely the~$8$ twists
$f\otimes\chi$ for $\chi\in \Ggroup$.  After base change to
$K' = \QQ(\zeta_5,\sqrt{-7})$ to trivialize all the characters, by
Proposition \ref{prop:ribetpyle}(d) we find that
\begin{equation}
  L(A,K',s) =L(S,K',s)^4 = L(f_{K'},s)^8
\end{equation}
so that
\begin{equation}
L(S,K',s)=L(f_{K'},s)^2.
\end{equation}

\section{Base change to \texorpdfstring{$\QQ(\sqrt{-7})$}{QQsqrtm7}}\label{sec:basechange}

In light of the results of the previous section, we now move to $K=\Q(\sqrt{-7})$.  
By Proposition \ref{prop:ribetpyle}(b), we have
\begin{equation} \label{eqn:qz8}
A_K \sim  \Czero^2,
\end{equation}
where $\Czero$ is a simple abelian fourfold defined over~$K$ with
$\End(C_0)_\Q \simeq \QQ(\zeta_8)$.

We now consider the base-change~$f_K$ of the newform~$f$ to~$K$, a Bianchi modular form of weight $2$.  In the classical language of automorphic forms, the base change is known as the Doi--Naganuma lift \cite{Doi-Naganuma} and can be obtained using the theta function method as generalized to arbitrary weight, level, and character by Friedberg \cite[Theorems 3.1, 3.2]{Friedberg-lift}.   More generally, base change can be understood in the language of automorphic representations: see Gerardin--Labesse \cite{Gerardin-Labesse} or Langlands \cite{Langlands-basechange}.  

  Although we do not use it explicitly, we recall the formula for the
  Hecke eigenvalues of $f_K$.  If $p$ splits in $K$ and $\frakp$ lies over $p$ then
$a_{\frakp}(f_K) = a_p(f)$; if $p$ is inert in $K$ then $a_{(p)}(f_K) =
a_p(f)^2-2\varepsilon(p)p$. 

Recall the character $\psi$ of order~$8$ defined in Section~\ref{sec:characters}.

\def\chiK{\left.\chi_{\sigma}\right|_K}

\begin{proposition}
  The base-change $f_K$ has the following properties.
  \begin{enumalph}
  \item We have $f_K \in S_2(\N, \psi^{-2})$, so $f_K$ has level 
    $\N=(175)=(5)^2(\sqrt{-7})^2$ and nebentypus character
    $\left.\varepsilon\right|_K = \psi^{-2}$.
  \item The Fourier coefficients of~$f_K$ lie in~$\QQ(\zeta_8)$.  
  \item For each inner twist $(\sigma, \chi)$ of~$f$, the base change $f_K$ has inner
    twist by~$(\sigma,\chi|_K)$; so $f_K$ has four distinct inner twists
    (including the trivial inner twist).
  \end{enumalph}
\end{proposition}
\begin{proof}
The level of $f_K$ is determined uniquely by the properties that it is
stable under Galois conjugation and of norm $N^2/\Disc(K)^2=1225^2/7^2=30625$ \cite[Lemma 5.2]{Turcas}, where
we use the fact that the $-7$-twist of $f$ has the same level by
Proposition~\ref{prop:inner-twist}: $f$ has inner twist
by~$(\sigma_{17},\chi_{-7})$.  Alternatively, since $1225=5^2 7^2$ the level of $f_K$ divides $(5)^2(\sqrt{-7})^2$ but this is already equal to the conductor of the character $\varepsilon|_K=\psi^{-2}$ of $f_K$.  (This also follows from the recognition of $f_K \otimes \psi \in S_2(\Gamma_0(\frakN))$ in Proposition \ref{prop:rational-twist} below.)

Let $(\sigma, \chi)$ be any of the inner twists of~$f$, so
$\sigma(f)=f\otimes\chi$.  Restricting to~$K$ gives
$\sigma(f_K) = f_K\otimes\chi|_K$.
Taking the inner twist to be $(\sigma_{17}, \chi_{-7})$ shows that
$\sigma_{17}(f_K) = f_K$, so the coefficients of~$f_K$ lie in the
fixed field of $\sigma_{17}$, which is~$\QQ(\zeta_8)$.  
\end{proof}

After restriction to~$K$, the characters in $\Ggroup$ coincide in pairs
to give all the even powers of~$\psi$.  Hence we have
\begin{equation}
L(A, K,s) = \prod_{i=1}^{4}L(f_K\otimes\psi^{2i},s)^2,
\end{equation}
and hence from \eqref{eqn:qz8} that
\begin{equation} \label{eqn:CzeroKs}
L(\Czero,s) = \prod_{i=1}^{4}L(f_K\otimes\psi^{2i},s).
\end{equation}

\section{Twisting}\label{sec:twisting}

We now investigate twists of the base change form $f_K \in S_2(\N,\psi^{-2})$ from the previous section. Define 
\begin{equation}
F \colonequals f_K \otimes\psi \quad \text{and} \quad G \colonequals f_K \otimes\psi^5 = F \otimes \psi^{4}.  
\end{equation}

\begin{proposition}\label{prop:rational-twist}
The twists $F,G \in S_2(\frakN)$ have rational eigenvalues; the form
$F$ has LMFDB label \textup{\lmfdbbmf{2.0.7.1}{30625.1}{c}}, while $G$
has label \textup{\lmfdbbmf{2.0.7.1}{30625.1}{e}}.
\end{proposition}

\begin{proof}
  Note that twisting leaves the level unchanged and the twisted forms have trivial nebentypus.
  
  It is enough to show that $F$ is fixed by
  all~$\sigma\in\Gal(\QQ(\zeta_8)\,|\,\QQ)$, since twisting by the
  quadratic character~$\psi^4$ preserves rationality.  Let
  $(\sigma,\chi)$ be an inner twist of $f$.  Then $\sigma(f_K)=f_K \otimes \chi|_K$ and
  \begin{equation}
  \sigma(f_K\otimes\psi) = \sigma(f_K) \otimes \sigma(\psi) = f_K\otimes \chi|_K\sigma(\psi).
\end{equation}
Since there is an inner twist for every~$\sigma\in\Gal(\QQ(\zeta_8)\,|\,\QQ)$,
the result is equivalent to establishing the character relations
\begin{equation} \label{eqn:psichi}
\psi = \chi|_K \sigma(\psi)
\end{equation}
for all such $\chi$.  Inspection of Table~\ref{table:innertwists} shows that for $a\in(\ZZ/8\ZZ)^\times$,
\begin{equation} 
(\chi_5^{(a-1)/2})|_K= (\psi^{-2})^{(a-1)/2} = \psi/\psi^a = \psi/\sigma_a(\psi)
\end{equation} 
verifying \eqref{eqn:psichi}.

Having identified a base change, we now examine the space of
Bianchi newforms of level~$\N$, weight~$2$ and trivial nebentypus
character, a space with LMFDB label \lmfdbbmfspace{2.0.7.1}{30625.1}.  The
new subspace has dimension~$60$ and contains~$8$ newforms with
rational Fourier coefficients.  Of these one matches~$F$, namely
\lmfdbbmf{2.0.7.1}{30625.1}{c}, while $G$ matches
\lmfdbbmf{2.0.7.1}{30625.1}{e}, and is the $\sqrt{5}$-twist of~$F$
(since $\psi^4=\left.\chi_5^2\right|_K$, the quadratic character
associated to $K(\sqrt{5}) \supseteq K$).
\end{proof}

\begin{remark} \label{rmk:loeffler}
It is straightforward to identify whether a Bianchi newform $F$ is a twist of base change since if $F$ and $\tau(F)$ are twist-equivalent then either $F$ is a twist of base change or is induced from a Hecke character \cite{Lapid-Rogawski}. Furthermore, for two Bianchi newforms $F_1$ and $F_2$ with trivial nebentypus, if $a_\frakp(F_1)^2 = a_\frakp (F_2)^2$ for all but finitely many primes $\frakp$ then $F_1$ and $F_2$ are twist-equivalent \cite{Ramakrishnan}.
\end{remark}

\begin{remark}
In fact, all eight rational Bianchi newforms in the space $S_2(\N)$
(with LMFDB label \lmfdbbmfspace{2.0.7.1}{30625.1}) give interesting
examples: we briefly describe the other six below (Section
\ref{sec:others175}).  The LMFDB does not contain details of Bianchi
newforms with nontrivial nebentypus character, such as the form $f_K$
itself.
\end{remark}

That the twisted forms $F$ and $G$ have rational eigenvalues means
that one can reasonably expect the corresponding twist of the fourfold
$C_0$ to split as a product of QM-surfaces.  Recall (see after
\eqref{eqn:qz8}) that $\End(C_0) \simeq \Q(\zeta_8)$; we choose an
automorphism $\xi \in \Aut(C_0)$ of order $8$.  With this choice, we
may twist $C_0$ via $\psi$ (relative to $\xi$), to obtain a new
fourfold defined over~$K$,  denoted~$C$:
\begin{equation}
C \colonequals C_0^{(\psi)}.
\end{equation}

\begin{theorem}\label{T:4fold}
There is an isogeny $\Cprime \sim S_F \times S_G$ defined over $K$, where $S_F,S_G$ are abelian surfaces defined over $K$ with QM by the rational quaternion algebra $B$ of discriminant $15$, satisfying $L(S_F,s) = L(F,s)^2$ and $L(S_G,s) = L(G,s)^2$.
\label{thm:decomposition}
  \end{theorem}

\begin{proof}
Let $\ell \equiv 1 \pmod{8}$ be prime and let $V_\ell C_0$ be the $\ell$-adic Tate representation of $C_0$, a free $\Q_\ell[\xi]$-module of rank $2$ where $\xi^4=-1$; choosing a primitive eighth root of unity $\zeta_8 \in \Q_\ell$, we may write 
\begin{equation} \label{eqn:VellC0}
V_\ell C_0 \simeq \bigoplus_{k \in (\ZZ/8\ZZ)^\times} W_{\ell,k}
\end{equation}
where each $W_{\ell,k}$ is a $\Q_\ell$-vector space of dimension $2$ with $\xi$ acting by the scalar $\zeta_8^{k}$.  

Now $V_\ell C_0$ is also a representation of $\Gal_K$; comparing with the $L$-series decomposition \eqref{eqn:CzeroKs}, as Galois representations we have (up to the choice of $\zeta_8$)
\begin{equation} 
W_{\ell,k} \simeq W_{\ell,1} \otimes \psi^{k-1}.
\end{equation}
On the twist, we have the Galois representation
\begin{equation} 
V_\ell C = V_\ell C_0 \otimes_{\Q_\ell[\xi]} \Q_\ell[\xi](\psi) 
\end{equation}
where $\Q_\ell[\xi](\psi)$ is $\Q_\ell[\xi]$ where $\Gal_K$ acts via $\psi$: explicitly, on the eigenspace of $\Q_\ell[\xi]$ where $\xi$ acts by $\zeta_8^k$, now $\Gal_K$ acts via $\psi^k$.  In terms of the splitting \eqref{eqn:VellC0} above, we have
\begin{equation} 
V_\ell C \simeq \bigoplus_{k \in (\ZZ/8\ZZ)^\times} W_{\ell,k} \otimes \psi^k
\simeq \bigoplus_{k} W_{\ell,1} \otimes \psi^{2k-1} 
\end{equation}
and thus
\begin{equation} \label{eqn:akfgh}
\begin{aligned}
  L(\Cprime,K,s) &= L(f_K\otimes\psi,s)L(f_K\otimes\psi^5,s)L(f_K\otimes\psi^9,s)L(f_K\otimes\psi^{13},s)\\
  &= L(F,s)^2L(G,s)^2.
  \end{aligned}
\end{equation}

Let $L \supseteq K'=K(\zeta_5)$ be the cyclic degree 8 extension of $K$ cut out by $\psi$.  By work of Kida \cite[p.~53]{Kida} we have 
\begin{equation} \label{eqn:c07si0}
\Res_{L|K} C_{0,L} \sim_K \prod_{i=0}^7 C_0^{(\psi^i)};
\end{equation}
i.e., the restriction of scalars of the base change of $A$ to $L$ is isogenous over $K$ to the product of the twists of $A$ by powers of $\psi$.  From the $L$-series decomposition \eqref{eqn:CzeroKs} and the theorem of Faltings, we conclude that 
\begin{equation} \label{eqn:c07si}
\prod_{i=0}^7 C_0^{(\psi^i)} \sim_K C_0^4 \times C^4. 
\end{equation}

Let $S$ be the building block from Proposition \ref{prop:ribetpyle}(d); up to isogeny (over $K'$), we may choose $S$ so that $\End(S) = \calO$, where $\calO \subseteq B$ is a maximal order.  Then $A_L \simeq C_{0,L}^2 \simeq S_L^4$, so $C_{0,L} \sim_L S_L^2$.  Since $\Q(\sqrt{2})$ splits $B$ and all maximal orders in $\calO$ are conjugate, there exists $\alpha \in \calO$ such that $\alpha^2=2$; then $\begin{pmatrix} 0 & -1 \\ 1 & \alpha \end{pmatrix} \in \M_2(\calO)$ is an element of order $8$, so we may consider twists of $S^2$ by powers of $\psi$.  Applying the theorem of Kida now to $S^2$, with \eqref{eqn:c07si0}--\eqref{eqn:c07si} we conclude there are isogenies over $K$
\begin{equation}
\prod_{i=0}^7 (S^2)^{(\psi^i)} \sim \Res_{L|K} S_L^2 \sim  \Res_{L|K} C_{0,L} \sim \prod_{i=0}^7 C_0^{(\psi^i)} \sim C_0^4 \times C^4.
\end{equation}
Since $C_0$ is simple, we conclude that $S$ is an isogeny factor of
$C$ over $K$, and $C \sim S \times S'$ for another abelian surface
$S'$.  We have $(S^2)^{(\psi^4)}=(S^{(\psi^4)})^2$ since
$S^{(\psi^4)}$ is the quadratic twist by $\psi^4$ (corresponding to the quadratic extension $K(\sqrt{5}) \supseteq K$),
so $S^{(\psi^4)}$ is also an isogeny factor of $C$ over $K$.  Since $S$ is a QM abelian surface,
its $L$-series is a square, so by \eqref{eqn:akfgh} we have $S
\not\sim S'$, and so $C \sim S \times S^{(\psi^4)}$, and thus up to
labelling $L(S,s)=L(F,s)^2$ and the result follows.
\end{proof}

\begin{remark}
The argument in Theorem \ref{T:4fold} can evidently be generalized to more general situations.
\end{remark}

\section{An explicit model for the building block}\label{sec:genus2curve}

In light of Theorem \ref{T:4fold}, we now seek explicit models for the
abelian surfaces $S_F$ and $S_G$ whose $L$-functions match $F$ and
$G$.  Because $G$ is a quadratic twist of $F$, it suffices to exhibit
a model for $S_F$.  We refer to Buzzard \cite{Buzzard-Shim} and Voight
\cite[Chapter 43]{Voight-book} as general references on abelian
surfaces with quaternionic multiplication.

Let $B \colonequals \QA{3}{5}{\Q}$ be the quaternion algebra
over $\Q$ of discriminant $15$, and let $\calO \subseteq B$ be a maximal order.
Then there exists $\mu \in \calO$ such that $\mu^2+15=0$, a 
\defi{principal polarization} on $\calO$.
The moduli space for principally polarized abelian surfaces 
with QM structure by $(\calO,\mu)$ is a Shimura curve $X_0(15,1)$ defined over $\Q$ by
work of Shimura \cite{Shimura} 
and Deligne \cite{Deligne}.  

An explicit model for $X_0(15,1)$ was computed by Jordan
\cite[Proposition 3.2.1]{Jordan} (see also Jordan--Livn\'e \cite[Table
1]{JordanLivne}) and Elkies \cite[(76)]{Elkies}): it is described by the
equation
\begin{equation} 
  X_0(15,1) \colon (x^2+3)(x^2+243)+3y^2 = 0. 
\end{equation}
In particular, $X_0(15,1)$ is a genus $1$ curve with $X_0(15,1)(\Q)=X_0(15,1)(\RR)=\emptyset$.  The Jacobian of $X_0(15,1)$ is the elliptic curve over $\Q$ with
LMFDB label~$\lmfdbec{15}{a}{5}$.

The quotient by the Atkin--Lehner involution $w_{15}$ gives the following map.
\begin{equation}
\begin{aligned}
X_0(15,1) &\to X_0(15,1)/w_{15} \simeq \PP^1 \\
j(x,y) &= x^2/81
\end{aligned}
\end{equation}  
The Igusa--Clebsch invariants of weights $2,4,6,10$ for a coarse space for the universal abelian surface over $X_0(15,1)/w_{15}$ are given by $(I_2:I_4:I_6:I_{10})$ where
\begin{equation} \label{I2610}
\begin{aligned}
I_2 &= 12(j^4+15j^3+105j^2+125j+10), \\
I_4 &= 45(4j+1)(j+3)^2(j-1)^4, \\
I_6 &= 9(84j^5+1414j^4+8865j^3+11157j^2+3895j+185)(j+3)^2(j-1)^4, \\
I_{10} &= 2(j+3)^6(j-1)^{12}.
\end{aligned}
\end{equation}
The invariants \eqref{I2610} were computed by Elkies (2007) using K3
surfaces and by Guo--Yang \cite[Table 4]{GuoYang} using Borcherds
products.

With this data in hand, we look for points on $X_0(15,1)$ over $K$.  The base change of its Jacobian to $K$ has rank $1$, so we will find infinitely many points.  Searching in a small box, we find the points $(x,y)=(\pm 9/\sqrt{-7},\pm 180/7)$ for which $j(x,y)=-1/7$ and normalized invariants $(648:23625:17474625:-392000000)$.  From these invariants, we construct a genus $2$ curve $X'$ over $K=\Q(\sqrt{-7})$ well-defined up to twist---we confirm that the Euler factors match \emph{up to sign} in the sense that $L_\frakp(X',T)=L_\frakp(f,\pm T)^2$ for good primes $\frakp$, and this gives strong evidence that we have found a match.  

Next, we identify the quadratic field $K_\delta \supseteq K$ which is associated to the twist.  For a good prime $\frakp$, write $L_\frakp(X',T)=(1-b_\frakp T + \mathrm{Nm}(\frakp)T^2)$ so that $b_\frakp = \mu(\frakp) a_\frakp(F)$ and $\mu(\frakp)=\pm 1$.  Then $\mu(\frakp)=-1$ if and only if $\frakp$ is inert in $K_\delta$; searching for fields with this property supported at primes dividing the discriminant of $X'$ gives a twist $X$ for which the Euler factors match for all $\frakp$ with $\mathrm{Nm}(\frakp) \leq 200$.  

In this way, we obtain a candidate model $X$ for a curve whose Jacobian putatively corresponds to $F$; however, this model has \emph{enormous} coefficients.  So we seek to reduce these coefficients.  We first apply the reduction algorithms of Bouyer--Streng \cite{Bouyer-Streng} to reduce the discriminant.  We would like to apply the Cremona--Stoll algorithm \cite{cremona-stoll} to reduce the coefficients of $X$ further, but the current implementation requires a real place in the base field.  Nevertheless, we can apply some reductions by hand, and then iterate.  In this way, we find the following model for $X=X_F \colon y^2 = f_F(x)$ where 
\begin{equation} \label{eqn:woahnelly}
\begin{aligned}
f_F(x) \colonequals &\ (-12774794511065444310316373455693222445855\sqrt{-7}\\ &\hspace{2cm}- 7695069988816582237324746836913139396789)x^6\\ & \quad+ 
    (36903356111088112749998450776069205963943\sqrt{-7} \\ &\hspace{2cm}+ 74766360407432704741728853365227303755813)x^5\\ & \quad+
    (-32191715349438139105751177726148896397205\sqrt{-7}\\ &\hspace{2cm}- 160594143536856131940416520163791514260770)x^4\\ & \quad+ 
    (3138757306334269455206726101294212247520\sqrt{-7}\\ &\hspace{2cm}+ 143070480545726568206381517716426697366880)x^3\\ & \quad+ 
    (8723717170068020100441878642537682595835\sqrt{-7}\\ &\hspace{2cm}- 59213195766173130554503672005222304866040)x^2\\ & \quad+ 
    (-4080652149555552661387182770620477619703\sqrt{-7}\\ &\hspace{2cm}+ 10457773570874119512625130761813845233833)x\\ & \quad+ 
    527840106295482417952525742290110846145\sqrt{-7}\\ &\hspace{2cm}- 493713158051445534935948551740658676743.
\end{aligned}
\end{equation}
By construction, we know that $J_F \colonequals \Jac(X_F)$ has $\End(J_F) \supseteq \calO$.  Looking at Euler factors we see that $J_F$ does not have CM, so $\End(J_F)=\calO$, since $\calO$ is a maximal order.  

\begin{proposition} \label{prop:FaltSer}
For all good primes $\frakp$ we have $L_\frakp(X_F,T)=L_\frakp(J_F,T)=L_\frakp(F,T)^2$,
\end{proposition}

\begin{proof}
  The proof of this proposition is a standard application of the method of Faltings--Serre, as explained by Dieulefait--Guerberoff--Pacetti \cite{DGPmodularity}; we only sketch some pertinent details here.
  Attached to $F = f_K \otimes \psi$ is a Galois representation $\rho_{F,2}\colon \Gal_K\to \GL_2(\Q_2^{\textup{al}})$, 
The characteristic polynomial of Frobenius at primes above $2$ and $11$ are given by $1\pm T + 2T^2$ and $1-4T + 11T^2$, whose roots are different and do not add zero, hence the coefficient field of $\rho_{F,2}$ may be taken as $\QQ_2$ by Taylor \cite[Corollary 1]{MR1253207}.  Since $J_F$ has QM (defined over $K$) by $B$ of discriminant $15$ and $2 \nmid 15$, the $2$-adic Tate representation $V_2 J_F$ is a free of rank $1$ over $B \otimes \Q_2 \simeq \M_2(\Q_2)$ and so affords a Galois representation $\rho_{J_F,2} \colon \Gal_K \to \GL_2(\Q_2)$.  

The $2$-division field $K(J_F[2])$ is defined by the polynomial 
\[ y^{12} - 6y^{11} + 21y^{10} - 50y^9 + 69y^8 - 42y^7 - 69y^6 + 210y^5 - 190y^4 + 48y^3 + 168y^2 - 160y + 64 \]
(as an absolute extension); we confirm this extension is an $S_3$-extension of $K$, so the mod $2$ representation $\overline{\rho}_{J_F,2}\colon \Gal_K \to \GL_2(\FF_2) \simeq S_3$ is surjective.    

Looking at Euler factors, the mod $2$ representation
$\overline{\rho}_{F,2}\colon \Gal_K \to \GL_2(\FF_2) \simeq S_3$
attached to $F$ is also surjective.  Since the conjugate of $F$ under
$\Gal(K\,|\,\Q)$ is a twist of $F$, the extension $K(J_F[2])$ is
Galois over $\Q$ with Galois group $S_3 \rtimes C_2 \simeq D_6$, and
so arises as the Galois closure of a sextic extension of $\Q$
unramified away from $2,5,7$.  From the Jones--Roberts database
\cite{JonesRoberts}, there are exactly $5$ such extensions whose
Galois closure contains $\Q(\sqrt{-7})$:
\begin{align*}
&    x^6 - 3x^5 + 2x^4 - 6x^3 + 25x^2 - 19x + 8, \\
&    x^6 - x^5 - x^4 - x^3 + 2x + 2, \\
&    x^6 - x^5 + 9x^4 - 11x^3 + 16x^2 + 14x + 4, \\
&    x^6 - 2x^5 + 11x^4 - 8x^3 + 30x^2 + 24x + 8, \\ 
&    x^6 - x^5 + x^4 - 29x^3 + 8x^2 - 64x + 512
\end{align*}
Looking at Euler factors, we rule out all but the first one, and then confirm that its Galois closure is isomorphic to the one for $J_F$; this shows that the two mod $2$ representations $\overline{\rho}_{F,2}$ and $\overline{\rho}_{J_F,2}$ are equivalent.  

Since both residual representations are isomorphic with absolutely
irreducible image, following Dieulefait--Guerberoff--Pacetti
\cite[Section 2.1]{DGPmodularity} we then show that fully the $2$-adic
representations $\rho_{F,2}$ and $\rho_{J_F,2}$ are equivalent by
showing there is no obstruction to lifting.  Computing a set of
obstructing primes (whose Frobenius classes generate the Galois group
of the maximal exponent $2$ extension of $K(J_F[2])$ unramified away
from $2,5,7$), we conclude it is enough to check that $\mytr
\rho_{F,2}(\Frob_\frakp) = \mytr \rho_{J_F,2}(\Frob_\frakp)$ for a
primes~$\frakp$ of $K$ dividing $3, 11, 13, 17, 23, 29$.  Having
already checked this for all good primes $\frakp$ up to norm
$\Nm(\frakp) \leq 200$, we conclude the proof.
\end{proof}

\section{Other rational Bianchi newforms at level (175)} \label{sec:others175}

As well as the newforms $F$ and~$G$ with labels
\lmfdbbmf{2.0.7.1}{30625.1}{c} and \lmfdbbmf{2.0.7.1}{30625.1}{e},
there are six other rational newforms of the same level, which display
a variety of phenomena.  All the quadratic twists mentioned here are
by $\sqrt{5}$ unless otherwise specified.
\begin{itemize}
  \item Newform \lmfdbbmfshort{2.0.7.1}{30625.1}{a} is the base-change of
    classical modular forms in $S_2(1225)$ with LMFDB labels
    \lmfdbcmf{1225}{2}{a}{b} and \lmfdbcmf{1225}{2}{a}{d}.  These are
    $-7$-twists of each other, and associated to the elliptic curves
    in the isogeny classes \lmfdbecisog{1225}{b} and
    \lmfdbecisog{1225}{d}.  The base-change to~$K$ of the curves in
    both these isogeny classes all lie in the isogeny class
    \lmfdbecnfisog{2.0.7.1}{30625.1}{a}.
  \item Newform \lmfdbbmfshort{2.0.7.1}{30625.1}{b} is a base-change, but
    of a newform with coefficients in $\QQ(\sqrt{2})$.  The associated
    modular abelian variety is a surface which splits over~$K$ into
    the product of elliptic curves, with associated elliptic curves 
    in the isogeny class
    \lmfdbecnfisog{2.0.7.1}{30625.1}{b}, which are $\QQ$-curves but not
    base-change.
  \item Newform \lmfdbbmfshort{2.0.7.1}{30625.1}{d} is also base-change,
    but of a newform with coefficients in $\QQ(\sqrt{5})$ whose
    associated modular abelian variety is a surface which does
    \emph{not} split over~$K$.  Hence there is no elliptic curve
    associated to the newform, while there is an abelian surface with
    QM.  The newform has CM by $-35$ (so is its own quadratic twist)
    and is an example of the situation described in
    \cite[p.~411]{JCextratwist}.
  \item Newform \lmfdbbmfshort{2.0.7.1}{30625.1}{f} is the quadratic twist
    of \lmfdbbmfshort{2.0.7.1}{30625.1}{a} so has associated elliptic
    curves \lmfdbecnfisog{2.0.7.1}{30625.1}{f} which are base-change.
  \item Newform \lmfdbbmfshort{2.0.7.1}{30625.1}{g} is the quadratic twist
    of \lmfdbbmfshort{2.0.7.1}{30625.1}{b} so has associated elliptic
    curves \lmfdbecnfisog{2.0.7.1}{30625.1}{g} which are $\QQ$-curves but
    not base-change.
  \item Newform \lmfdbbmfshort{2.0.7.1}{30625.1}{h} is the quadratic twist
    of \lmfdbbmfshort{2.0.7.1}{1225.1}{a} at level $\N'=(35)$.  It has
    associated elliptic curves in isogeny class
    \lmfdbecnfisog{2.0.7.1}{30625.1}{h} which are base-changes of
    those in isogeny classes \lmfdbecisog{1225}{h} and
    \lmfdbecisog{1225}{j}, and is itself the base-change of classical
    newforms \lmfdbcmf{1225}{2}{a}{h} and \lmfdbcmf{1225}{2}{a}{j}.
\end{itemize}

\section{Relation with the Paramodular Conjecture}\label{sec:paramodular}

We conclude with an application to the Paramodular Conjecture.  Recall
the following conjecture due to Brumer--Kramer \cite{BrumerKramer} and
its corrigendum \cite{BrumerKramer2}, as amended by Calegari.

\begin{conjecture} \label{conj:BK0}
There is a bijection between the set of isogeny classes of QM abelian fourfolds $B$ over $\Q$ of conductor $N^2$ and the set of cuspidal, nonlift Siegel paramodular newforms $f$ of genus $2$, weight $2$, and level $N$ with rational Hecke eigenvalues, up to nonzero scaling.  Moreover, if $B \leftrightarrow f$ in this bijection, we have the equality
\begin{equation} \label{eqn:LBsLfs}
L(B,s)=L(f,s,\spin)^2. 
\end{equation}
\end{conjecture}

As with Bianchi modular forms, the frequently arising case is where
$\End(B)_\Q \simeq \M_2(\Q)$, in which case $B \sim A^2$ for $A$ an
abelian surface over $\Q$ with $\End(A)=\ZZ$ with
$L(A,s)=L(f,s,\spin)$.  (When $B$ is a division algebra, the
nonexistent geometric object $A'$ such that $L(A',s)=L(f,s,\spin)$
could be thought of as a \emph{fake abelian surface}.)

To our Bianchi newform $F$, by theta lift one can attach a Siegel
paramodular form $\Pi$, with rational eigenvalues \cite[Theorem
  4.1]{Bergeretal} such that $L(F,s) = L(\Pi,s,\text{spin})$.  The
form $\Pi$ satisfies the hypothesis of the Paramodular Conjecture
(Conjecture \ref{conj:BK0}), so there should be a QM abelian fourfold
attached to $\Pi$.  Indeed, by Theorem \ref{thm:decomposition} the
Weil restriction of scalars $\Res_{K|\QQ} S_F$ of $S_F$ from $K$ to
$\Q$ gives the desired geometric object, and
\begin{equation} \label{eqn:lressf}
L(\Res_{K|\QQ} S_F,s)=L(S_F,K,s)=L(F,s)^2 = L(\Pi,s,\spin)^2. 
\end{equation}

\begin{remark}
In work of Boxer--Calegari--Gee--Pilloni \cite[Lemma
  10.3.2]{Boxeretal}, assuming standard conjectures a general argument
is given to explain the existence of a geometric object attached to an
automorphic form on $\GSp_4$ with rational Hecke eigenvalues.  In
concrete examples, the veracity of the standard conjectures is very
hard to check; however, when both the automorphic form and (fake)
abelian surface can be explicitly given, the method of Faltings--Serre
can often be applied successfully in practice
\cite{BrumerPacetti_et_al} to establish an equality of $L$-functions,
such as \eqref{eqn:lressf}.
\end{remark}

\bibliographystyle{alpha}

\providecommand{\NOOPSORT}[1]{}

\end{document}